\documentclass {amsart}
\usepackage{amssymb,amsthm,color, amscd, mathtools}
\begin{document} 
\title{\bf Abelian $p$--groups and their endomorphism rings}

\author {Phill Schultz}
\address[phill.schultz@uwa.edu.au]{The University of Western Australia}

\theoremstyle{plain}
\newtheorem{theorem}{Theorem}[section]

\newtheorem{corollary}[theorem]{Corollary}
\newtheorem{lemma}[theorem]{Lemma}
\newtheorem{proposition}[theorem]{Proposition}
\newtheorem{conjecture}[theorem]{Conjecture}
\newtheorem{hypothesis}[theorem]{Hypothesis}
\newtheorem{condition}[theorem]{Condition}
\newtheorem{fact}[theorem]{Fact}
\newtheorem{problem}[theorem]{Problem}

\theoremstyle{definition}
\newtheorem{definition}[theorem]{Definition}
\newtheorem{notation}[theorem]{Notation}

\theoremstyle{remark}
\newtheorem{remark} [theorem]{Remark}
\newtheorem{remarks}[theorem]{Remarks}
\newtheorem{example}[theorem]{Example}
\newtheorem{examples}[theorem]{Examples}

\renewcommand{\leq}{\leqslant}
\renewcommand{\geq}{\geqslant}
\newcommand{\Aut}{\mathop{\mathrm{Aut}}\nolimits}
\newcommand{\End}{\mathop{\mathrm{End}}\nolimits}
\newcommand{\Ker}{\mathop{\mathrm{Ker}}\nolimits}
\newcommand{\Hom}{\mathop{\mathrm{Hom}}\nolimits}

\newcommand\bbQ{{\mathbb{Q}}}
\newcommand\bbZ{{\mathbb{Z}}}
\newcommand\bbN{{\mathbb{N}}}
\newcommand\bbT{{\mathbb{T}}}
\newcommand{\bbP}{\mathbb{P}}
\newcommand{\bbS}{\mathbb{S}}
\newcommand{\bbSN}{\mathbb{S}\mathbb{N}}
\newcommand{\bbF}{\mathbb{F}}

\newcommand{\boldi}{\mathbf{\iota}}
\newcommand{\boldn}{\mathbf{n}}
\newcommand{\boldb}{\mathbf{b}}
\newcommand{\boldc}{\mathbf{c}}
\newcommand{\bolde}{\mathbf{e}}
\newcommand{\bolda}{\mathbf{a}}
\newcommand{\boldd}{\mathbf{d}}
\newcommand{\bi}{\mathbf{\iota}}
\newcommand{\boldm}{\mathbf{m}}
\newcommand{\boldt}{\mathbf{t}}
\newcommand{\boldu}{\mathbf{u}}
\newcommand{\boldk}{\mathbf{k}}
\newcommand{\boldell}{\mathbf{\ell}}
\newcommand{\boldsigma}{\mathbf{\sigma}}
\newcommand{\boldtau}{\mathbf{\tau}}

\newcommand{\calQ}{{\mathcal Q}}
\newcommand{\calP}{\mathcal{P}}
\newcommand{\calC}{\mathcal{C}}
\newcommand{\calD}{\mathcal{D}}
\newcommand{\calA}{\mathcal{A}}
\newcommand{\calH}{\mathcal{H}}
\newcommand{\calS}{\mathcal{S}}
\newcommand{\calB}{\mathcal{B}}
\newcommand{\calX}{\mathcal{X}}
\newcommand{\calY}{\mathcal{Y}}
\newcommand{\calL}{\mathcal{L}}
\newcommand{\calM}{\mathcal{M}}
\newcommand{\calG}{\mathcal{G}}
\newcommand{\calT}{\mathcal{T}}
\newcommand{\calI}{\mathcal{I}\kern-1pt\textit nd}
\newcommand{\calW}{\mathcal{W}}
\newcommand{\calJ}{\mathcal{J}}
\newcommand{\calE}{\mathcal{E}}
\newcommand{\calK}{\mathcal{K}}
\newcommand{\calF}{\mathcal{F}}
\newcommand{\calV}{\mathcal{V}}
\newcommand{\calU}{\mathcal{U}}
\newcommand{\calN}{\mathcal{N}}
\newcommand{\calZ}{\mathcal{Z}}
\newcommand{\calLN}{\mathcal{LN}}
\newcommand{\calR}{\mathcal{R}}
\newcommand{\calFI}{\mathcal{FI}}
\newcommand{\calId}{\mathcal{I}\kern-1pt\textit d}

 \newcommand{\up}{^{(\xi)}}
\newcommand{\ov}{\overline}

\newcommand{\rank}{\operatorname{rank}}
\newcommand{\order}{\operatorname{order}}
\renewcommand{\Im}{\operatorname{Im}}
\newcommand{\type}{\operatorname{type}}
\newcommand{\lcm}{\operatorname{lcm}}
\newcommand{\trace}{\operatorname{trace}}
\newcommand{\Rad}{\operatorname{Rad}}
\newcommand{\ord}{\operatorname{order}}
\newcommand{\Jon}{\operatorname{Jon}}
\newcommand{\Sym}{\operatorname{Sym}}
\newcommand{\Ext}{\operatorname{Ext}}
\newcommand{\ACD}{\operatorname{ACD}}
\newcommand{\im}{\operatorname{im}}
\newcommand{\height}{\operatorname{height}}
\newcommand{\stem}{\operatorname{stem}}
\newcommand{\gap}{\operatorname{gap}}
\newcommand{\expp}{\operatorname{exp}}
\newcommand{\h}{\operatorname{height}}
\newcommand{\Ann}{\operatorname{Ann}}
\newcommand{\length}{\operatorname{length}}
\newcommand{\branch}{\operatorname{branch}}
\newcommand{\grad}{\operatorname{grad}}
\newcommand{\lift}{\operatorname{lift}}
\newcommand{\coexp}{\operatorname{co--exp}}
\newcommand{\Ulm}{\operatorname{Ulm}}
\newcommand{\di}{\operatorname{\dagger-{\rm inv}}}
\newcommand{\ind}{\operatorname{ind}}
\newcommand{\Range}{\operatorname{Range}}
\newcommand{\rot}{\operatorname{root}}
\newcommand{\Mono}{\operatorname{Mono}}
\newcommand{\Iso}{\operatorname{Iso}}
\newcommand{\divi}{\operatorname{div}}
\newcommand{\GL}{\operatorname{GL}}\
\newcommand{\Mod}{\operatorname{Mod-}}
\newcommand{\depth}{\operatorname{depth}}
\newcommand{\tr}{\operatorname{trace}}
\newcommand{\lat}{\operatorname{lat}}

\newcommand{\la}{\langle} 
\newcommand{\ra}{\rangle}      

\newcommand{\Sub}{\texttt{Subgroups}}
\newcommand{\QSub}{\texttt{Quasi-Subgroups}}
\newcommand{\Bases}{\texttt{Bases}}
\newcommand{\bases}{\texttt{Bases}}
\newcommand{\QBases}{ \texttt{Quasi-Bases}}
\newcommand{\Subgroup}{\texttt{Subgroup}}

\newcommand{\st}[1]{\la{#1}\ra_*}

\begin{abstract}  The lattice of fully invariant subgroups of an abelian $p$--group and the lattice of ideals of its endomorphism ring are classified by   systems of cardinal invariants.
\end{abstract}
\noindent
 \maketitle

\section{Introduction}

Finite abelian groups were  the first algebraic structures to be studied from the abstract point of view.  Every such group is  a direct sum of finite $p$--groups for finitely many primes $p$; a finite $p$--group $G$ is a direct sum of   cyclic $p$--groups, say $m_n$ summands isomorphic to $\bbZ(p^n)$ for finitely many $n$, and the set of such pairs $(n,\,m_n)$ of positive integers is a complete isomorphism invariant for $G$.  
 
This result can be extended  a little further: let $G$ be a  direct sum of cyclic $p$--groups. Let  \[S=\{n\in \bbZ\colon \bbZ( p^n)\text{ is isomorphic to a cyclic summand of } G\}\]    and let $m_n$ be the number of summands of order $p^n$. Both $S$ and $m_n$   may be   infinite. Then the set of pairs $\{(n,\,m_n)\colon n\in S\}$ is a complete isomorphism invariant of $G$.

 Every abelian $p$--group  decomposes as a direct sum of a unique maximal divisible summand and  a reduced complement unique up to isomorphism. Divisible  $p$--groups have a satisfactory structure theorem: they are direct sums of quasi--cyclic groups $\bbZ(p^\infty)$ and the number of such summands  is an invariant. Thus the classification of torsion abelian groups   reduces to the problem of classifying unbounded reduced $p$--groups.

In the first  half of this paper, I review known results to show that every unbounded  reduced abelian $p$--group $G$ is determined up to isomorphism by an infinite ordinal $\lambda$ and for each limit ordinal $\xi<\lambda$, a  $\Sigma$--cyclic  group $B^{(\xi)}$ and a group $H^{(\xi)}$ which lies between $B^{(\xi)}$ and its torsion completion in the $p$--adic topology.. The  groups $B^{(\xi)}$ have complete systems of cardinal invariants, but the extensions $H^{(\xi)}$ remain mysterious unless $G$ is simply presented. 

The second half of the paper is a description  of the lattice of fully invariant  subgroups of an abelian $p$--group  and its application to the classification of the    lattice of ideals of the endomorphism ring of the group.  
 \section{ Notation} 
Except when explicitly stated otherwise, from now on the noun \lq group\rq\  signifies  reduced abelian $p$--group. Function names are written on the right of their arguments.
Otherwise, the notation  throughout this paper is standard, as found for example in \cite{Fuchs}.  
In particular,
 
\begin{itemize}\item  $\bbN$ denotes the poset of natural numbers including 0, $\bbN^+$ the positive integers, and for $n\in\bbN^+, \ [n]=\{1,\dots,s\}$;

\item For $n\in \bbN^+,\ \bbZ(p^n)$ is the cyclic $p$--group of order  $p^n$;

\item For a group $G$ and $a\in G$, the \textit{exponent  of $a$},\ $\exp(a)=\min\{k\in\bbN\colon p^ka=0\}$; $G$ is bounded
if $ \sup\{\exp(a)\colon a\in G\}$ is finite and if so, this supremum is denoted $\exp(G)$;

\item  $G[p^k]=\{a\in G\colon \exp(a)\leq k \}$; in particular, $G[p]$ is  the socle of $G$;
\item $p^kG[p]$ denotes $(p^kG)[p]$;
\item For a subgroup $H$ of $G$, denoted $H\leq G,\ \rank(H)$ is the cardinality of a maximum integrally independent subset of $H$;
\item $\calE=\calE(G)$ is  the endomorphism ring of $G$;  
 
\item $\bbF_p$ is the field of $p$ elements so $G[p]$ is a $\bbF_p$--space of dimension $\rank(G[p])$;
\item An ordinal $\mu$ is a successor if there exists an ordinal $\sigma$ such that $\mu=\sigma+1$ and a limit  otherwise; in particular 0 is a limit ordinal.
 
\end{itemize}
\section{Invariants of $G$}

The following results are independent of the choice of prime $p$.

\subsection{Basic subgroup} 
A \textit{basic group} $B$ is a  direct sum of cyclic groups,    together with a presentation of the form $B=\bigoplus_{i<s} B_i$ where \begin{enumerate}\item $s\in\bbN^+$ or $s=\omega$;
\item $B_i\cong (\bbZ(p^{n_i}))^{m_i},\ n_i\in\bbN^+$ and $m_i$ is a non--zero cardinal;
 \item $i<j<s$ implies $ {n_i} <n_j$. \end{enumerate} 
 
 We  denote $B$ by $B(\boldn,\,\boldm)$. where $\boldn=(n_i\colon i<s)$ and $\boldm=(m_i\colon i<s)$.

 Let $G$ be a group. By  \cite[Chapter 5, \S 5]{Fuchs}, $G$ has  a  {basic subgroup} $B$  satisfying the following properties:
\begin{enumerate}  \item[B1] $B$ is a basic group, pure in $G$, and for all $i<\omega,\ \bigoplus_{j\leq i}B_j$ is a maximal pure $p^{n_i}$--bounded subgroup of $G$;

\item[B2]  For all $i<\omega,\ G=\bigoplus_{j\leq i}B_j\oplus\left(\bigoplus_{j>i}B_j +p^iG\right)$;
\item[B3]  $\rank(B_i)=m_i=\rank(B_i[p])$ so $\rank(B)=\sum_{i<\omega}m_i=\dim(B[p])$.
\item[B4]  $G/B$ is divisible;
\item[B5]  If $A$ is a pure subgroup of $G$ which is a direct sum of cyclic groups, then $A$  is a direct summand of some basic subgroup of $G$.
\item[B6]  For all $n\in\bbZ^+,\   \ p^{n}G[p]\cong B_{n+1}[p]\oplus p^{n+1}G[p]$. 
\item[B7]  All basic subgroups of $G$ are isomorphic.

\end{enumerate}

Property  B2  implies that $G$ is bounded if and only if $G=B$, in which case $G$ itself is  basic. More precisely, if $G$ has exponent $e$, then $p^nG=0$ for all $n\geq e$.  Consequently,  $G$ is unbounded  if and only if $B$ is unbounded.

Property  B4  ensures   that every $f\in\calE$ is determined by its value on $B$  i.e., if $f,\ g\in\calE$ satisfy  $g|_B=f|_B$ then $f=g$.  In particular, if $f$ is zero on $B$, then $f=0$.
Thus  by B1 and B4, $G$ is a pure extension of $B$ by $G/B$, both of which are classified by cardinal invariants. 

Property B6 ensures that if $A$ is a basis for the $\bbF$--space $p^kG[p]/p^{k+1}G[p]$, then there is an independent set  $\{b_a\colon a\in A\}$  in $B_{k}$ such that for all $a\in A,\ a=p^{k-1} b_a$. 

Property B7 ensures that  every basic subgroup of $G$ is  isomorphic to $B(\boldn,\,\boldm)$ for a unique pair $(\boldn,\,\boldm)$. To summarise:

 \begin{proposition}\label{basic} Let $G$ be a group.\begin{enumerate}\item  If $G$ is bounded, then there exists $k\in\bbN^+$ and   finite sequences $\boldn=(n_i\in \bbN^+\colon i\in[ k],\  n_i<n_{i+1})$ and $\boldm=(m_i\colon i\in[ k],\  m_i\text{ a cardinal})$ such that $G\cong B(\boldn,\,\boldm)$;  
 
  \item If   $G$ is unbounded and there are infinite sequences $\boldn=(n_i\colon i<\omega,\ n_i\in\bbN^+)$ and $\boldm=(m_i\colon i<\omega,\ m_i\text{ a cardinal })$ such that $G$ has basic subgroup $B(\boldn,\,\boldm)$. \qed  \end{enumerate}
  
  \end{proposition}

\subsection{Heights and Ulm invariants} 
The following properties, due to Ulm,  are found in  \cite[Chapters 10 and 11]{Fuchs}: 

 Let  $p^0G=G$ and for all ordinals $\kappa>0$,\[ p^\kappa G=\begin{cases} p(p^{\kappa-1}G)\text{ if $\kappa$ is a successor}\\
\bigcap_{\mu<\kappa}p^\mu G\text{ if $\kappa$ is a limit}\end{cases}\]

 Since $(p^\kappa G\colon \kappa\geq 0)$ is a  decreasing chain of subgroups, there  is a least ordinal $\lambda$, called  the \textit{length of $G$ } for which $p^\lambda G= p^{\lambda +1}G$. 
 If $G$ is reduced,   $p^\lambda G=0$. Otherwise $p^\lambda G$ is the maximal divisible subgroup of $G$.
 
 For each $a\in G$, the \textit{height of $a$}, $\height(a)$, is the least $\kappa$ for which $a \in p^\kappa G\setminus p^{\kappa+1}G$ if such $\kappa$ exists; otherwise  $a\in p^\lambda G$ and $\height(a)=\infty$. 

The subgroups $p^\kappa G,\ \kappa<\lambda$, have the following properties: 

\begin{enumerate}\item[U1]   $\left(p^\kappa G[p]\colon \kappa\leq\lambda\right)$ is a  decreasing  chain of $\bbF_p$--spaces which are   fully invariant subgroups of $G$;

\item[U2]   For all $\kappa < \lambda,\ V_\kappa= p^\kappa G[p]/p^{\kappa+1}G[p]$ is  an $\bbF_p$--space such that $p^\kappa G[p]\cong V_\kappa\oplus p^{\kappa+1}G[p]$.   Denote 
 the dimension $\dim(V_\kappa)$ by  $u_\kappa$. The    cardinals $u_\kappa$  are called  the \textit{Ulm invariants} of $G$. 
 
 \item[U3]   If $G$ and $H$ are countable groups with the same Ulm invariants, then $G\cong H$.
 
 \item[U4]   The property   B2 implies that the finite Ulm invariants $u_n,\ n<\omega$, of $G$ are also the   Ulm invariants of any basic subgroup of $G$;
 
 \item[U5]   Let $B=B(\boldn,\,\boldm)
 $ be a basic subgroup of $G$. Then for all $j<\omega,\ u_j=m_{j}$; in particular,   $B$, and hence $G$,  has a cyclic summand of exponent $i+1$ if and only if $u_{i}\ne 0$.
 \end{enumerate}
 \begin{notation}Let $G$ be a group of length $\lambda$.
 
 \begin{enumerate} \item Let $\Omega=(\kappa<\lambda\colon u_{\kappa}\ne 0)$, and let  $\boldu(G)=(u_\kappa\colon \kappa\in\Omega)$ be the sequence of non--zero Ulm invariants of $G$.

\item   Let $\boldu=(u_\kappa\colon \kappa\in\Omega)$ be a  sequence of non--zero cardinals. $\boldu$ is \textit{ $\lambda$--admissible} if there is a reduced $p$--group $G$ of length $\lambda$ such that $\boldu=\boldu(G)$.
  \item An abelian $p$--group  $G$   is \textit{simply presented} (\cite[Chapter 11, \S 3]{Fuchs}) if it has a presentation $G=\la X, R\ra$ where every relation in $R$ has the form $px=0$ or $px=y$. The simply presented groups have the following  properties:
 
 \begin{enumerate}\item Basic groups are simply presented;
 \item Countable simply presented groups are basic; 

\item A  simply presented group of limit length is the direct sum of simply presented groups of smaller length; in particular, a simply presented group of length $\omega$ is basic;
 \item The class of simply presented groups is closed under direct sums and summands;
 \item If $G$ is simply presented, then so are $p^\kappa G$ and $G/p^\kappa G$ for all ordinals $\kappa$;
 \item If $G$ and $H$ are simply presented groups with the same Ulm invariants, then $G\cong H$.  
 \end{enumerate}
  \end{enumerate}
 \end{notation}
 Necessary and sufficient conditions for admissibility of Ulm invariants are proved in  \cite[Chapter 11, Theorem 3.7]{Fuchs}.
 We quote them here for future reference.

  \begin{proposition}\label{fuchs} Let $\boldu=(u_\kappa\colon \kappa\in\Omega)$ be a  sequence of non--zero cardinals and let $\lambda=\sup\{\kappa\colon \kappa\in\Omega\}$. 
 
  \begin{enumerate}\item  If   $\boldu$ is $\lambda$--admissible, then
 for all $\rho<\lambda$ and   $\kappa$ with  $\kappa+\omega<\lambda$,\[ \sum_{\rho\geq\kappa+\omega}  u_{\rho} \leq \sum_{n<\omega}u_{\kappa+n} \tag  {*}\]

 \item Let   $\boldu=(u_\kappa\colon \kappa\in\Omega)$ satisfy the condition  $(*)$.  Then there exists a simply--presented group $G$  of length $\lambda$ such that for all  $\kappa<\lambda, \     u_\kappa$ is the $\kappa$--Ulm invariant of $G$, i.e. $\boldu$ is $\lambda$--admissible.
\qed\end{enumerate}\end{proposition}

  Condition $(*)$ is called the \textit{Ulm sum criterion}.  
  \subsection{The $p^\kappa$--adic topology \cite{J-L}}
  
  Let $G$ be a group of length $\lambda$ and let $\kappa\leq\lambda$ be a limit  ordinal. The subgroups $(p^\mu G\colon \mu\leq\kappa)$ form a basis of neighbourhoods of 0 for a topology on $G$, called the \textit{$p^\kappa$--adic topology} which make $G$ a topological group which is Hausdorff if and only if $\kappa=\lambda$. When $\kappa=\omega$, the topology is  simply called the \textit{$p$--adic topology}.
   Since endomorphisms of $G$ do not decrease height, every $f\in\End(G)$ is continuous in the $p^\kappa$--topology for all $\kappa\leq\lambda$.
  
  If $G$ is Hausdorff in the $p^\kappa$--topology, $G$ has a completion which in general is not a $p$--group, although its torsion subgroup is.  
  
  \section{Separable groups}  
  
   In this section, we describe the groups of length $\omega$,  i.e., unbounded groups with no elements of infinite height, called \textit{separable groups}.

 Let $G$ be a separable group with basic subgroup $B=\bigoplus_{i<\omega}B_i$. 
  Then the $p$--adic  completion of $B$ is a subgroup of $\prod_{i<\omega}B_i$ whose  torsion subgroup    $\ov B$    is a $p$--group consisting of   all bounded sequences in   $\prod_{i<\omega}B_i$,   called the \textit{torsion completion   of $B$}. If $|B|=\mu$, then $|\ov B|=2^\mu$ and $\ov B/B$ is divisible of rank $2^\mu$.

 Fuchs \cite[Chapter 10]{Fuchs} shows   that    the following properties are equivalent;. 
 \begin{enumerate} \item[S1] $G$ is separable;
 \item[S2]  every finite subset of $G$ is contained in a finite direct summand,  
 \item[S3] $B\leq G\leq \ov B$, $G$ is pure in $\ov B$ and $\ov B/G$ is divisible.
 \end{enumerate}
Consequently,  a separable   $G$ with basic subgroup $B$ is   simply--presented if and only if  $G=B$.
  The importance of separable groups lies in the fact that, as I   show in \S 5, they are the intrinsic building blocks of all reduced $p$--groups.

Throughout this section, unless noted otherwise, $G$ is a separable group with basic subgroup $B=\bigoplus_{i<\omega}B_i=B(\boldn.\boldm)$.

Countable separable groups are basic,   but    uncountable separable groups  are notoriously complicated; as Fuchs says,   \cite[page 316]{Fuchs},  no general structure theorem is available, and none is expected. 
In particular,    Corner \cite{Corn} shows that even in the simplest case when each $B_n\cong \bbZ(p^n)$ there exists a family of $2^{2^{\aleph_0}}$ non--isomorphic pure subgroups   of $\ov B$ containing $B$. Corner's proof \cite[Theorem 2.1]{Corn} is non--constructive in the sense that it shows that there are $2^{2^{\aleph_0}}$ non--isomorphic rings each of which is isomorphic to the endomorphism ring of some group $G$ with $B\leq G\leq \ov B$.   Hence there are as many non--isomorphic  groups $G$ in $\ov B$.
The result holds for any $|B|\leq 2^{\aleph_0}$ and   for larger $|B|$   Shelah has shown, \cite[p. 332]{Fuchs}, that for every infinite cardinal $\kappa$ there are $2^\kappa$ non--isomorphic separable groups   of cardinal $\kappa$ with the same basic subgroup.

 Although a full classification of  separable groups is out of reach, it is possible to classify an important subclass. 
 
  \begin{definition} Let $B$ be an unbounded basic group. A group $G$ satisfying $B\leq G\leq \ov B$ is an \textit{almost  complete extension of $B$} if there exists a decomposition $B=C\oplus D$ where  $C$ and $D$ are  unbounded, no cyclic summand of $C$ is isomorphic to a cyclic summand of $D$,
and $G=\ov C\oplus D$, where $\ov C$ is the torsion completion of $C$.\end{definition}

In the following lemma, $B=\bigoplus_{i\in\bbN}B_i$ and we represent elements of $\ov B$ as bounded sequences $\boldb=(b_i\colon i\in\bbN)\in\prod_{i\in\bbN}B_i$ where $b_i\in B_i$.

\begin{proposition}\label{alcom} Let $G=\ov C\oplus D$ and $G'=\ov{C'}\oplus D'$ be almost complete extensions of $B$. Then $G$ and $G'$ are pure subgroups of $\ov B$ and $G\cong G'$ if and only if $C=   C'$. 
\end{proposition}

\begin{proof} Since $\ov G=\ov B=\ov G',\ G$ and $G'$ are subgroups of $\ov B$ containing $B$. Let $n\in\bbN$ and $\boldb\in\ov B$ such that $p^n\boldb\in G$. Then $p^n\boldb=p^n\boldc+p^n\boldd$ with $p^n\boldc\in\ov C$ and $p^n\boldd\in D$ Since both $\ov C$ and $D$ are pure in $\ov B$, there exist $\bolde\in \ov C$ and $f\in D$ such that $p^n\bolde=p^n\boldc$ and $p^nf=p^n\boldd$. Hence $G$ is  pure in $\ov B$ and similarly $G'$ is pure in $\ov B$.

 Let $\alpha\colon G\to G'$ be an isomorphism and let $B_i$ be a maximal homogeneous summand of $B$ contained in $C$ and $B_j$ a mavimal homogeneous summand of $B$ contained in $D$.  Since $\alpha$ preserves exponents and heights, $B_i\alpha$ is a maximal summand of $G'$ generated by elements of exponent $n_i$ and height 0. Hence $B_i\alpha=B_i$ and similarly for $B_j\alpha=B_j$. Thus $C'=C\alpha= C$ and $D'=D\alpha= D$. Since $\alpha$ is continuous in the $p$--adic topology, $\ov C\alpha=\ov {C'}$, so $  C=  C'$. 

The converse is clear, so we have the following classification:
\end{proof}

\begin{corollary}\label{weaker} Let $B=\bigoplus_{i\in\bbN}B_i$ be an unbounded basic group.  
  There exist $ {2^{\aleph_0} }$ non--isomorphic  almost complete extensions of $B$  parametrized by partitions of $\bbN$ into two infinite subsets.

\end{corollary}
\begin{proof} 
 Proposition \ref{alcom} shows that almost complete extensions of $B$ are pure subgroups of $\ov B$ and 
 that non-isomorphic almost complete extensions derive from unequal decompositions of $B$. Hence There is a 1--1 correspondence between almost complete extensions of $B$ and partitions of $\bbN$ into two infinite sets;

  $\bbN$ has $2^{\aleph_0}$ subsets of which $\aleph_0$ are finite, and $\aleph_0$ have a finite complement. Hence there remain  $2^{\aleph_0}$  which determine partitions of $\bbN$ into two infinite subsets.
 \end{proof}
. 
 \section{Ulm subgroups}

    Throughout this section,  $G$ is a reduced group of  infinite length $\lambda$ with basic subgroup $B$; $\boldu= (u_\kappa\colon \kappa\in\Omega)$ is the sequence of non--zero Ulm invariants of $G$, and 
  $\Xi$ is  the   sequence of  limit ordinals $0\leq \xi < \lambda$.
  
For all $\xi\in\Xi$, let $G^{(\xi)}$  denote $p^\xi G=\bigcap_{\mu<\xi} p^\mu G $. Let $B^{(\xi)}$ be a basic subgroup of $G^{(\xi)}$ so that   $G^{(0)}=G$ and $ B^{(0)}=B$. 
   The group   $G^{(\xi)}$ is called the \textit{$\xi$--Ulm subgroup}  of $G$, (\cite[p. 305]{Fuchs}), and   $B^{(\xi)}$ the \textit{$\xi$--derived basic subgroup} of $G$.  
   
   The sequence $(G^{(\xi)}\colon \xi\in\Xi)  $ is called  the \textit{Ulm subgroup sequence} and   $(B^{(\xi)}\colon \xi\in\Xi)$ the \textit{derived basic subgroup sequence} of $G$.
   
     \begin{lemma}\label{Xi}   Let  $\Omega,\ \Xi,\  (G^{(\xi)})$ and $(B^{(\xi)})$ be defined as above.
   \begin{enumerate}\item   $ \Xi\subseteq\Omega$, that is, for all $\xi\in\Xi,\ u_\xi\ne 0$;
   
   \item The  sequence  $(G^{(\xi)}\colon \xi\in\Xi)$ is strictly decreasing;

   \item If $\xi+\omega\leq\lambda$, then    $B^{(\xi)}$ is unbounded;

   \item If $\xi+\omega>\lambda$, then there exists $k\in\bbN$ such that $\lambda=\xi+k$ and $B^{(\xi)}$ has exponent $k$;
 \end{enumerate}
 \end{lemma}
 \begin{proof} (1) Let $\xi\in\Xi$. Since $\xi<\lambda,\ G^{(\xi)}\ne 0$, so  $p^\xi G[p]\ne 0$ and hence $\xi \in\Omega$.

(2)  By definition, $G^{(\xi+1)}$ is a proper subgroup of $G^{(\xi)}$ and $G^{(\xi+\omega)}$ is a proper subgroup of $G^{(\xi+n)}$ for all $n<\omega$ whenever these goups exist..

(3) Since   $ G^{(\xi+\omega)}=\bigcap_{\nu<\xi+\omega}p^\nu G,\  G^{(\xi+\omega)}\leq\bigcap_{n<\omega}p^n G^{(\xi)}$. Hence for each $0\ne a\in G^{(\xi)}$ there is a sequence $(n_i)\in\bbN$ and a sequence $(b_i) \in G^{(\xi)}$ such that $ a=p^{n_i}b_i$ where $b_i$ has height 0 so $\exp(a)\geq n_i$.  

(4) The fact that $\lambda=\xi+k,\ k\in\bbN^+$, implies that $B^{(\xi)}$ is has exponent $k$.
 \end{proof}

    The connection between Ulm subgroups and the Ulm invariants described in \S 3.2 is as follows:

 \begin{proposition}\label{fin} Let $\xi\in\Xi$ and $n\in\bbN$.  
  \begin{enumerate}\item  For all $a\in G^{(\xi)}[p],\ p^na\in G^{(\xi)}$   if and only if $\height(a)\geq\xi+n$;

\item If $\lambda=\xi+n$ then $B^{(\xi)}$ has exponent $n$;
 
   \item If $\xi+\omega<\lambda$, then 
    $B^{(\xi)}$ is unbounded and  isomorphic to a basic subgroup of $G^{(\xi)}/G^{(\xi+\omega)}$.
      
 \item For all $\xi\in\Xi,\ \rank(B^{(\xi)})\geq\sum_{\xi<\rho \in\Xi}\rank(B^{(\rho)})$.

  \end{enumerate}
  \end{proposition}
   \begin{proof} (1)   The following are equivalent:
   \begin{itemize}\item[(a)]  $p^na\in G^{(\xi)}$;
   \item[(b)]   there exists $b\in G^{(\xi)}$ of height $\xi$ such that $p^nb=a$;
   \item[(c)] $\height(a)\geq\xi+n$.
   \end{itemize}
   
   (2) By definition, $n$ is minimal such that $p^nG^{(\xi)}=0$.
 
  (3) Since $G^{(\xi+\omega)}\ne 0$, $G^{(\xi +n_i)}\ne 0$ for some sequence $\boldn=(n_i)$. Let   $m_i$   be the dimension of  $p^{n_i}G^{(\xi)}[p]/p^{n_{i+1}}G^{(\xi)}[p],\ \boldm=(m_i)$ so $B^{(\xi)}=B(\boldn,\boldm)$ is unbounded.
   
  Let $\eta\colon G\to G/G^{(\xi+\omega)}$ be the natural epimorphism. Then $\eta$ maps $B^{(\xi)}$ isomorphically onto a basic subgroup of   $G^{(\xi)}/G^{(\xi+\omega)}$.

 (4) By the Ulm sum criterion (Proposition \ref{fuchs}), for all $\xi\in \Xi,\ \rank(B^{(\xi)})=\sum_{n<\omega} u_{\xi+n}\geq\sum_{\rho>\xi}u_\rho=\sum_{\xi<\rho \in\Xi}\rank(B^{(\rho)})$.
  \end{proof}

\begin{corollary}\label{graded}   With the notation above, the  graded group $\bigoplus_{\xi\in\Xi}G^{(\xi)}/G^{(\xi+\omega)}$ is   separable. It has  a    basic subgroup $\bigoplus_{\xi\in\Xi}B^{(\xi)} $ which is a direct sum of unbounded basic groups satisfying for all $\xi\in\Xi,\ \rank(B^{(\xi)})\geq\sum_{\xi<\rho \in\Xi}\rank(B^{(\rho)})$.
\end{corollary}

\begin{proof}

 Lemma \ref{Xi} (2) implies that the decreasing sequence  $\left(G^{(\xi)}\right)$ determines the direct sum $\bigoplus_{\xi\in\Xi}G^{(\xi)}/G^{(\xi+\omega)}$ of which each summand is separable.

The rest follows from Proposition \ref{fin} (3) and (4). 
\end{proof}

Let $\lambda$ be an infinite ordinal and $\Xi$ the sequence of infinite ordinals $\xi<\lambda$. 
 A sequence $(B^{(\xi)}\colon  \xi\in\Xi)$ of basic groups is called \textit{admissible}  if for all $\xi\in\Xi,\ \rank(B^{(\xi)})\geq\sum_{\xi<\rho \in\Xi}\rank(B^{(\rho)})$. 
 
 \begin{proposition}\label{missible} Let $(B^{(\xi)}\colon  \xi\in\Xi)$ be an admissible sequence of basic groups.
Define a sequence $(\boldu_\kappa)$ of cardinals by $u_{\xi+k}=\rank(B^{({\xi})}_k)$ for all $\xi\in\Xi$ and all $k<\omega$. Then $\boldu$ is the $G$--admissible Ulm invariant for a simply presented group $G$ whose derived Ulm sequence is $(B^{(\xi)})$.
\end{proposition}

\begin{proof} The definition of admissible sequence of basic groups shows that $\boldu$ satisfies the Ulm sum condition of Proposition \ref{fuchs} (1). Hence there exists a simply presented group $G$ whose Ulm sequence is $\boldu$ and  derived Ulm sequence is $(B^{(\xi)})$.
 \end{proof}

  Proposition \ref{fin} and Corollary \ref{graded} show that every admissible Ulm sequence determines an 
admissible sequence of basic groups.  The following theorem proves the converse.

\begin{remarks}\begin{enumerate}\item Theorem \ref{MAIN} is essentially a rewording of 
\cite[Chapter 10, Theorem 1.9]{Fuchs} where the rank condition of Corollary \ref{graded} is replaced by an equivalent cardinality condition.

\item Let $(B^{(\xi)}\colon \xi\in\Xi)$ be an admissible sequence of basic groups. For all $\xi\in\Xi$ let $r_\xi=\rank (B^{(\xi)})$. By Corollary \ref{graded}, for all $(\sum_{\xi<\rho}r_\rho\colon \xi\in\Xi )$ is a decreasing sequence of cardinals. Hence it must stabilise at some finite index, i.e., there exists $\mu\in\Xi$ such that $\mu$ is the $n$th element of $\Xi$ and $r_\mu=\sum_{\kappa<\rho}r_\rho$ for all $\kappa\geq\mu$. 
\end{enumerate}\end{remarks}

\begin{theorem}\label{MAIN} 

\begin{enumerate}\item If $G$ is a group of length $\lambda$ then its sequence of derived basic groups is admissible.

\item If  $\calB=(B^{(\xi)}\colon  \xi\in\Xi)$  us any admissible sequence of basic groups, then there exists a simply presented group $G$ whose  sequence of basic groups is $\calB$. For all $\xi\in\Xi,\ G^{(\xi)}/ G^{(\xi+\omega)}\cong B^{(\xi)}$.
 
\end{enumerate}
\end{theorem}
\begin{proof} (1) is a consequence of 
 Proposition \ref{fin} (4).
 
 (2)  is a consequence of Corollary \ref{graded}  and Proposition \ref{missible}.  \end{proof}
 
 In the following Corollary,   \lq unique\rq\ means unique up to isomorphism 
  \begin{corollary}\label{theor} \begin{enumerate}\item Every group of length $\lambda$ determines a unique admissible sequence $\left(B^{(\xi)}\colon \xi\in\Xi\right)$ of basic groups.
  
\item  Every admissible sequence   of basic groups determines a unique simply presented group.

\item Let $G$ be a group of length $\lambda$.
   $G$ is determined uniquely by the sequence $\left(B^{(\xi)}\colon \xi\in\Xi\right)$ together with   for each $\xi\in\Xi$ a pure subgroup $H^{(\xi)}$ satisfying $B^{(\xi)}\leq H^{(\xi)}\leq\ov B^{(\xi)}$.
\qed\end{enumerate}
 
\end{corollary}

 \section{Indicators}The following definitions and properties are   found in \cite[Chapters 10 and 11]{Fuchs}.  
\begin{enumerate} 

\item An \textit{indicator}  is a   strictly  increasing  finite or countable sequence of ordinals followed by the symbol $\infty$, denoted  $\boldsigma=(\sigma_i\colon i\in[n],\infty)$ or $\boldsigma=(\sigma_i\colon i<\omega,\infty)$.   When $n=0$, the sequence is denoted $(\infty)$. An indicator $\boldsigma$ has \textit{  length $n$} in the former case,  and has \textit{finite entries} if  each $\sigma_i$ is finite.

\item The set $\calI$  of  indicators  is   ordered   pointwise:   $\boldsigma$ \textit{precedes} $\boldtau$, denoted $\boldsigma\preceq \boldtau$, if both have infinite length and $\sigma_i\leq\tau_i$ for all $i<\omega$, or $\boldsigma$ has length $\geq n,\ \boldtau$ has length $n$   and $\sigma_i\leq\tau_i$ for all $i\leq n$.

 In general $\calI$ has infinite descending chains, e.g., $(1,\infty )\succ (1,2 ,\infty)\succ \dots$ and infinite ascending chains, e.g., $(\kappa_1,\infty )\prec
(\kappa_2,\infty )\prec\dots$ where $(\kappa_i\colon i<\omega)$ is a countable increasing chain of ordinals.
\item  Let $\boldsigma$ be an indicator of length $\geq n$ and let $i<n$. We say that $\boldsigma$ \textit{has a gap at $i$} if $\sigma_i+1<\sigma_{i+1}$.

\item Let $G$ be a group. An indicator $\boldsigma\in\calI$ is \textit{$G$--admissible} if   for all $i<\omega$, if $\boldsigma$ has a gap at $i$ and $\sigma_i=\kappa$, then the   Ulm invariant  $u_\kappa\ne 0$.  Let $\calI(G)$ be the set of  $G$--admissible indicators.
\item  A set  $\calS$  of  indicators   satisfies \textit{Kaplansky's gap condition for $G$} if   $\calS\subseteq\calI(G)$.
\item If $G$ is bounded of exponent $n$ then $\calI(G)$ has a unique minimum, namely $(0,1,\dots,n-1,\infty)$.

\item Let $a\in G$ with $\exp(a)=n+1$. The \textit{indicator of $a$} is the $G$--admissible indicator $\ind(a)=(\sigma_i\colon i\in[n],\infty)$, where $\sigma_i=\height(p^ia)$, so $\sigma_0=  \height(a) $ and $\sigma_n=\height(p^na)$.

\item   Let $a\in G$ and $ f\in \calE$. Then $\ind(a)\preceq\ind(af)$.
\item $G$ is \textit{transitive} if for all $a,\,b\in G,\   \ind(a)\preceq\ind(b)$ implies there exists $f\in\calE$ such that $b=af$. 
Simply presented groups are transitive.
  \end{enumerate}
%%%%%%%%%%%%%%%%%%%%%%%%%%%%%%%%
 \section{Fully invariant subgroups}
 
     Let $G$ be a group with endomorphism ring $\calE$.  A subgroup $H\leq G$ is \textit{fully invariant} if for all $f\in\calE,\ Hf\leq H$. Denote by $\calH=\calH(G)$ the set of fully invariant subgroups of $G$. For example, $0,\,G,\, p^kG$ and $ G[p^n]\in\calH$ for all $\kappa<\length(G)$ and $n<\exp(G)$. 

\begin{lemma}\label{fiL} Let $G$ be a group  with endomorphism ring $\calE$ and $\calH$ its set of fully invaraint subgroups.     Then $\calH$ is a complete lattice under inclusion.
\end{lemma}
\begin{proof} Let $S\subseteq\calH$  and $f\in\calE$.  For all $a\in\bigcap S,\ af\in \bigcap S$. Hence $S\in\calH$.

Let $b\in \sum S$. Then there exists a finite set $S_i\colon i\in[n]$ such that $b=\sum b_i$ with $b_i\in S_i$. Hence $bf=\sum b_if\in \sum S$, so $\sum S\in\calH$. 
\end{proof}
     
 \begin{proposition}\label{lattice} Let $G$ be a group,    $\boldsigma=(\sigma_i,\infty)\in \calI(G)$,  and    $G(\sigma)=\{a\in G\colon \sigma \preceq\ind(a) \}$.
 
 \begin{enumerate}\item $G(\boldsigma)\in\calH(G)$;  
\item Under the pointwise ordering, $\calI(G)$ is a lattice;

\item Let $\tau$ be any finite segment of $\boldsigma$, followed by $\infty$. Then $\tau=\ind(a)$ for some $a\in G$;

\item  If   $\ind(a)=(\sigma_{\kappa_1},\dots,\sigma_{\kappa_n},\infty)$ then $\exp(a)=n+1$ and for all $i\in[0,n],\  p^ia\in p^{\kappa_i}G[p^{n-i+1}]$.
\item If $\boldsigma\preceq \tau\in \calI(G)$, then   $G(\tau)\leq G(\sigma)$.
\end{enumerate} \end{proposition}
 
 \begin{proof} (1) Since $\ind(0)=\{\infty\},\ 0\in G(\sigma)$. Let $a$ and $b\in G(\sigma)$. Since $\height(p^n(a-b))\geq \min\{p^na,\,p^nb\}$ for all $n<\omega,\ a-b\in G(\boldsigma)$ so $G(\sigma)\leq G$.
 By 4.2 (8), $G(\sigma)$ is fully invariant.

(2)  It is clear that $\calI(G)$ is a subposet of $\calI$. Let  $\sigma,\ \tau\in\calI(G)$. Suppose $\rho=\min\{\sigma,\,\tau\}$ has a gap at $\kappa$, Then either $\sigma_\kappa$ or $\tau_\kappa=\rho$ and hence has a gap at $\kappa$. Hence the $\rho_\kappa$--Ulm invariant of $G$ is non--zero so $\rho\in\calI(G)$.
 
 A similar argument shows that $\mu=\max\{\sigma,\,\tau\}\in\calI(G)$.
 
 (3) Let $\boldtau=(\sigma_\kappa,\sigma_{\kappa+1},\dots,\sigma_{\kappa+n},\infty)$. Then $\boldtau$ is a $G$--admissible indicator, so by \cite[Chapter 10, Lemma 1.1]{Fuchs},   $\boldtau=\ind(a)$ for some $a\in G$. 
 
 (4) By definition for all $i\in[0,n],\ \exp(p^ia)=n-i+1$ and $\height(p^ia)= \kappa+i$. Hnce $a\in p^{\kappa+i}G[p^{n-i+1}]$.
 
 (5) This follows  immediately from (4).
 \end{proof}
 
\begin{remark}\label{lat22} 

 $\boldsigma_1\ne\boldsigma_2$ does not imply that $G(\boldsigma_1)\ne G(\boldsigma_2)$.
\end{remark}

 \begin{notation} Let $G$ be a group and $\boldsigma\in\calI(G)$. We have seen in Proposition \ref{lattice} (1) that $G(\boldsigma)\in\calH$. An $H\in\calH(G)$  is called an \textit{indicator subgroup} if $H=G(\boldsigma)$ for some $\boldsigma\in\calI(G)$.

 A \textit{fundamental subgroup} of $G$ is one of the form $p^\kappa G[p^n]=(p^\kappa G)\cap (G[p^n])$,  for some $\kappa\leq\lambda,\ n\in\bbN$.   
  \end{notation}
\begin{lemma}\label{fuin} Let $G$ be unbounded of length $\lambda$.\begin{enumerate}\item For all $\kappa\leq\lambda$  and $n\in\bbN, \ p^\kappa G[p^n]\in\calH(G)$;
\item For all $n,\,m\in\bbN$ and all $\kappa,\,\mu <\lambda,\ p^\kappa G[p^n]\leq p^\mu G[p^m]$ if and only if $n\leq m$ and $\kappa\geq \mu$.
\end{enumerate}
\end{lemma}
\begin{proof} (1) Let $H=p^\kappa G[p^n],\ a\in H$ and $f\in\calE$. Since $f$ does not decrease heights nor increase exponents, $af\in H$.

(2) Let $a\in p^\kappa G[p^n]$ so $\exp(a)\leq n$ and $\height(a)\geq\kappa$. Since $a\in p^\mu G[p^m]$ if and only if  $\exp(a)\leq m$ and $\height(a)\geq\mu$, $p^\kappa G[p^n]\leq p^\mu G[p^m]$ if and only if $n\leq m$ and $\kappa\geq \mu$.
\end{proof}
 \begin{proposition}\label{fundam} Let $G$ be a group and $\calH=\calH(G)$. Let $S=\{p^\kappa G[p^n]\colon \kappa\in K,\ n\in N\}$ be a set of fundamental subgroups of $G$ for a set $K$ of ordinals $<\lambda$ and a set $N$ of natural numbers.   Then
 \begin{enumerate} \item $\bigcap_{H\in S} H$ and $\sum_{H\in S} H\in \calH$;
    \item Let    $\boldsigma=(\kappa_n\colon n<\omega, \infty)$ be a $G$--admissible indicator and  for all $ n<\omega$ let $H_n= p^{\kappa_n} G[p^n]$. Then $(H_n)$  is a strictly decreasing sequence in $\calH$. \item For each finite segment $\boldtau$ of $\boldsigma$  there is a fundamental subgroup $H$  containing $G(\boldtau)$.
 \end{enumerate}
 \end{proposition}
 \begin{proof}
  (1)  Follows from Lemma \ref{fiL}.

 (2)  Since $(\kappa_n)$ increases with $n$, by Lemma \ref{lattice}, $(H_n)$ is a strictly decreasing sequence in $\calH$. 
 
 Let $a\in G(\boldsigma)$ have exponent $m$ so $\ind(a)=(\tau_0,\dots, \tau_m,\infty)$ such that for some segment $(\sigma_\kappa,\dots \sigma_{\kappa+m}$ of $\boldsigma,\ \tau_i\geq \sigma_{\kappa+i}$.

Then $G(\boldsigma)\leq p^{\tau_0}G[p^m]\leq H_0$.
  \end{proof}
  \begin{remark} If $\boldsigma$ is any indicator,  the proof of Proposition \ref{fundam} (2) shows that $G(\boldsigma)\in\calH(G)$ provided that $G(\boldsigma)$ is non--empty. The $G$--admissibility is required only to  ensure $G(\boldsigma)\ne\emptyset$.
  \end{remark}
\begin{corollary}\label{label}\begin{enumerate}\item Every indicator subgroup is contained in a  fundamental subgroup;
\item If $G$ is simply presented, then every fully invariant subgroup is contained in a fundamental subgroup.
\qed\end{enumerate}\end{corollary}

  \subsection{The Fundamental Matrix, Bounded Case:}

 Let $  G=\bigoplus_{i\in [k]}\left(\bbZ(p^{n_i})\right)^{m_i}$, let $e=\exp(G) =n_k$,  and let $O=(n_i \colon i\in[k])$ be the sequence of exponents  of the homocyclic summands of $G$. The  Fundamental Matrix  of $G$ is the
  $e\times k $ matrix $M=M(G)$ with $(i,j)$--entry \[M(i,j)=p^{ j}G[p^{i}]=\{a\in G\colon \height(a)\geq  j\text{ and }\exp(a)\leq i\}.\]  
 
 The rows $i$ represent exponents and are indexed  bottom to top  from $[e]$; the   columns $j$ represent heights and are indexed from $[0,k-1]$    left to right. For example,
if $G$ is homocyclic of exponent $e$, then $k=1
$, each $i\in[e]$  and   $M$ is the column matrix \[\begin{bmatrix} M(e,0)\\M(e-1,0)]\\\vdots\\M(1,0)\end{bmatrix}=\begin{bmatrix} G\\ G[p^{e-1}]\\\vdots\\ G[p]\end{bmatrix}.\]

\begin{lemma}\label{lat2} With the notation above,
\begin{enumerate}
\item Every entry $M(i,j) $  of $M$    is a distinct fundamental subgroup of $G$;

\item The entries of $M$ form a lattice   with  $M(i,j)\vee M(k,\ell)=M(\max\{i,k\},\min\{j,\ell\})$ and $M(i,j)\wedge M(k,\ell)=M(\min\{i,k\},\max\{j,\ell\})$. 

\end{enumerate}
\end{lemma}

\begin{proof} (1) Each $p^jG[p^i] $ is a fundamental subgroup, so we must show that if $( i,j)\ne (i',j')$ then $M(i,j)\ne M(i',j')$. Since the columns of $M$ are indexed by the homocyclic components of $G$, only  the subgroup $p^iG[p^j]$ contains elements of minimum height $i$. Similarly only  $p^iG[p^j]$ contains elements of maximum exponent $j$.

(2)  Let $a\in M(i,j)\vee M(k,\ell)$. Then $\height(a)\geq\max\{i,k\}$ and $\exp(a)\leq\min\{j,\ell\}$. The proof for meets is similar.
\end{proof}

\begin{corollary}\label{admi1} 
  Each entry $M(i,j) $ determines a \lq quartering\rq\  of $M$ by the row coordinate $i$ and the column coordinate $j$. The entries of the south--east quadrant are   contained in $M(i,j)$, those in the north--west quadrant   contain  $M(i,j)$ and those in the other quadrants  are not comparable to $M(i,j)$.
\qed\end{corollary}
\begin{remark} For all $(i,j)\in \exp(G)\times \exp(G),\ p^jG[p^i]$ is a fully invariant subgroup of $G$    whereas the fundamental matrix uses only $(i,j)\in \exp(G)\times k$, where $k\leq \exp(G)$ is the set of indices of non--zero homocyclic components. Nonetheless, the Fundamental Matrix  contains an alias of every $p^jG[p^i]$. 
\end{remark}

\begin{proposition}\label{nonetheless} Let $s$ be the sequence of indices of homocyclic components of $G$. Let $p^jG[p^i]$, where  $j\not\in s$,   be a non--zero fundamental subgroup of $G$. Then there exists   $\ell\in s$ such that $p^jG[p^i]=p^\ell G[p^i]$.
\end{proposition}

\begin{proof} Since $p^jG[p^i]\ne 0$, there exists a least $\ell\in O$ such that $j<\ell$. Since $\height(a)\geq j$ mplies $\height(a)\geq\ell,\  p^jG[p^i]=p^\ell G[p^i]$.
\end{proof}

The  Fundamental Matrix $M$ also embodies $\calI(G)$ the lattice of $G$--admissible indicators as paths, in the following sense.  

\begin{notation}A  \textit{rising   path} in $M$ is sequence of entries  which form an ascending digraph with arrows  $\left(M(i,j),\ M(i+1, j+\ell)\right)$  where $\ell\in[exp(G)-j] $. The initial entry $M(i,j)$ is arbitrary, and each subsequent entry comes from the row above and some column to the right. Such a rising   sequence always exists, but has only the initial entry $M(i,j)$ if $i=\exp(G)$ or if $j=\exp(G)-1$. Note that if  $k> j+1$, then the Ulm sequence for $G$ has a gap at $i$, so rising paths satisfy Kaplansky's Gap Condition of \S 6.
\end{notation}
\begin{lemma}\label{embodies} Let $\boldsigma=(\sigma_1,\sigma_2,\dots,\sigma_n)$ be a $G$--admissible indicator, and let $a\in G$ with  $\ind(a)=\boldsigma$.% so $\exp(a)=n+1$.

 Then $M$ contains a  rising   path   $M(i,j_0),\ M(i+1,j_1),\dots,\ M(i+n-1, j_{n-1})$  such that $a\in
M(i,j_0),\ pa\in M(i+1,j_1),\dots$ and $p^na\in M(i+n-1, j_{n-1})$.

Conversely, let $(M(i,j_0),\ M(i+1,j_1),\dots, M(i+k-1,j_{k-1})$ be a  rising path   in $M$ and let $\boldsigma=(j_0,\,j_1,\dots,j_{k-1})$  . Then there exists $a\in G$ with $\ind(a)=\boldsigma$.
\end{lemma}
\begin{proof}     $p^ka\in M(i+k-1, j_{k-1})$ for all $k\in[n]$ if and only if $M$ admits the given  rising  path.
 \end{proof}

 Suppose  $\boldsigma\in\calI(G)$ corresponds to the   rising path   $\boldsigma=M(i,j_0),M(i+1,j_1),\dots,\\ M(i+m, j_{m})$. How is the fully invariant subgroup $G(\boldsigma)$ related to the entries $M(n+j, i_j)$?
 
 \begin{proposition}\label{relate} Let $\boldsigma= (\sigma_1,\dots, \sigma_m)$ and   let
 $(M(i+m, j_m)\colon m\in[k])$  be the corresponding rising   path.. Then  
 $G(\sigma)\leq M(i+1,j_1)\leq\cdots\leq M(i+k, j_k)$.
 \end{proposition}
 
 \begin{proof} Let $a\in G$. Then $a\in G(\boldsigma)$ if and only if for all $m\in[k],\ \height(p^m a)\geq j_m)$ if and only if for all  $m\in[j],\ a\in M(i+m,j_m)$.
 \end{proof}
 \begin{example}\label{bex} 
 
Let $G=\la a\ra\oplus \la b\ra$ with $\exp(a)=2$ and $\exp(b)=4$. 
The Fundamental Matrix of $G$ is:
\vskip 0.5cm

 	$M= \begin{bmatrix}  M(4,0)&M(4,1)\\
  M(3,0)&M(3,1)\\
  M(2,0)&M(2,1)\\
  M(1,0)&M(1,1)\\
  \end{bmatrix}$   
 	 $ =	  \begin{bmatrix}  \la a\ra\oplus\la b\ra&\la pa\ra\oplus\la pb\ra\\
 \la a\ra\oplus\la pb\ra&\la pa\ra\oplus\la pb\ra \\
 \la pa\ra\oplus\la p^2b\ra&\la pa\ra\oplus\la p^2b\ra \\
 \la pa\ra\oplus\la p^2b\ra&\la pa\ra\oplus\la p^3b\ra \\
  \end{bmatrix}$
 
  \bigskip  
 \end{example} 

By   brute force enumeration, $G$ has 11   fully invariant subgroups 
so 
11 admissible  rising diagonal paths, 3 of length 1, 4 of length 2, 2 of length 3 and 1 of lengths 0 and 4. 
\vskip 0.5cm

 \begin{tabular} {|c|c|c|}
 
 \hline
Indicator&FI Subgroup&Ind. Decomp\\

\hline
$(\infty)$&0&$0$\\

 $(1,\infty)$&G[p] &$\la pa\ra\oplus \la p^3b\ra$\\

 $(2,\infty)$&$p^2G $&$ \la p^2b\ra$\\
 
   $(3,\infty)$&$p^3G $&  $\la p^3b\ra$\\ 
 
$(0,1,\infty)$&$G[p^2]$&$\la a\ra\oplus \la p^2b\ra$\\

$(1,2,\infty)$&$pG $&$\la pa\ra\oplus \la pb\ra$\\

$(1,3,\infty)$&$pG[p^2]$&$\la pa\ra\oplus \la p^2b\ra$\\
  
$(2,3,\infty)$&$p^2G $&$\la p^2b\ra$\\

$(0,1,2,\infty)$&$G[p^3]$&$\la a\ra\oplus \la pb\ra$\\ 
 
$(1,2,3,\infty)$&$pG $&$\la pa\ra\oplus\la pb\ra$\\ 

 $ (0,1,2,3,\infty)$&$G[p^4]$&$G$\\
\hline
\hline

\end{tabular}
\vskip 0.5cm 

 \begin{corollary}\label{loi} Let $G$ be a bounded group. The lattice $\calI(G)$ of $G$--admissible indicators determines the  lattice $\calH$ of fully invariant subgroups of $G$. 
 \qed\end{corollary}

%%%%%%%%%%%%%%%%%%%%%%%%%%%%

\subsection{The Fundamental Matrix, Unbounded Case:}

Let $G$ be an unbounded group of length $\lambda$. 
The  Fundamental Matrix of $G$ is the
  $\bbN^+\times\lambda$ matrix $M=M(G)$ with $(n,\kappa)$--entry the fundamental subgroup \[M(n,\kappa)=p^\kappa G[p^n]=\{a\in G\colon \exp(a)\leq n\text{ and }\height(a)\geq\kappa\}\]  for all $n\in\bbN^+$ and $\kappa<\lambda$.  The rows are indexed by $\bbN^+$ bottom to top and   columns  by $\lambda$   left to right. 
  
  Many of the results of  the bounded case, such as Lemma \ref{lat2},   hold true in the unbounded case. For example the entries of the fundamental matrix  are fully invariant subgroups and the fundamental matrix contains the intersections of all finite sets of its entries .
  
Once again, the  Fundamental Matrix $M$   embodies the lattice  $\calI(G)$ of $G$--admissible indicators as follows.  A countable sequence of entries of $M$ of the form \[\boldsigma=   (M(n,\kappa _0),\ M(n+1,\kappa _1),\dots, M(n+k, \kappa _{k})\colon k<\omega,\kappa_k<\lambda)\] where $\kappa _0<\kappa _1<\dots< \kappa _{k-1}<\dots<\lambda$  is called a \textit{rising   sequence}. A rising sequence $\boldsigma$ is \textit{$G$--admissible} if whenever $ \kappa _k+1<\kappa _{k+1}$, the $\kappa_i$ Ulm invariant of $G$ is non--zero.

The initial entry $M(n,\kappa_0)$ is arbitrary, and each subsequent entry comes from the row above and some column to the right. Note that if $\kappa_0<\mu$ for a limit ordinal $\mu\leq\lambda$, then such an infinite rising  sequence always exists. 

\begin{lemma}\label{embodies1} Let $\boldsigma=(\sigma_0,\sigma_1,\dots,\sigma_\kappa,\dots,\infty)$ be a $G$--admissible indicator. Let $\tau=(\sigma_\kappa,\sigma_{\kappa+1},\dots,\sigma_{\kappa+n} )$ be a finite segment of $\boldsigma$ and let $\rho=(\rho_i\colon i\in\bbN^+)$ be a countable subsequence of $\bbN^+$.

 \begin{enumerate}\item There exists $a\in G$ with  $\ind(a)=\tau$;
\item  $ \left(M(\kappa,\sigma_\kappa),\ M(\kappa+1,\sigma_{\kappa+1}),\dots, M(\kappa+n, \sigma_{\kappa+n})\right)$ is a finite  admissible rising   sequence   in $M$. 

 \item  $M$ contains an admissible  rising   sequence $(M(i,\rho_i )\colon  i\in\bbN^+)$; 
 
\item Let  $\left(M(i,\tau_i \colon  i\in\bbN^+)\right)$ be an infinite  admissible rising  sequence   in $M$. Then $\boldsigma =(\tau_i\colon i\in\bbN^+,\infty)$ is a $G$--admissible indicator. 

\end{enumerate}\end{lemma}

\begin{proof} (1) Such elements $a$ of $G$ exist by Kaplansky's Theorem \cite[Chapter n10, Theorem 2.2]{Fuchs}.

(2) and (3) are immediate from the definition of   admissible rising sequence.

(4) follows from the fact that any subsequence of a $G$--admissible indicator is a  $G$--admissible indicator.
\end{proof}
 
The following theorem shows that the Fundamental Matrix determines all $H\in\calH$ of the form $G(\boldsigma)$ where  $\boldsigma$ is an admissible indicator.
:

\begin{theorem} Let $G$ have fundamental matrix $(M(i,j))$ and let $\boldsigma=(\sigma_i\colon i\in\bbN^+)$ be a $G$--admissible sequence.
Then $G(\sigma)=\sum_{i\in\bbN^+}M(i,\sigma_i)$.

\end{theorem}
\begin{proof} Let $a\in G$. Then $a\in G(\sigma)$ if and only if for all $i\in\bbN^+,\  \height(p^ia)\geq \sigma_i$ if and only if $a\in \sum_{i\in\bbN^+}H(i,\kappa_i)$.
\end{proof}

\begin{remarks} \begin{enumerate}\item Every $H\in\calH$ is $G(\boldsigma)$ for some $\boldsigma$ if and only if $G$ is transitive.

\item  $G$ is bounded iff $\calH$ is finite, but $\calH$ does not tell us whether $G$ is finite.

\item The length of longest chain and the width of  largest antichain in $\calH$ are finite if and only if $G$ is bounded.
\end{enumerate}
\end{remarks}

\section{The endomorphism ring of $G$}

Current knowledge of the structure of the endomorphism ring $\calE$ of an abelian  $p$--group $G$ is presented in \cite[Chapter 4, \S 20]{KMT} which mainly draws on the results presented in \cite{Fuchs} and \cite{Pierce}.
For the particular case of separable $G$, see \cite{Corn} and \cite{String}.
These references  are mainly concerned with identifying the Jacobson radical $\calJ$ of $\calE$ and the factor ring $\calE/\calJ$. The  main structure theorem  states that  if $G$ has non--zero Ulm invariants $(u_\kappa\colon \kappa<\lambda)$  then $\calE$ is  the split extension of the Jacobson radical  $\calJ$ of $\calE$ by a subring of the  direct product $\prod_{i<\lambda}\calL(u_i)$ where $\calL(u_i)$ is the ring of linear transformations of the $\bbF_p$--space of dimension  $u_i$.  $\calJ$ is charatacterised (\cite[Lemma 4.5]{Pierce}) as the $p$--adic closure of the ideal of endomorphisms which raise finite  heights in the socle,

The paper \cite[Theorem 4.5]{AvS} contains an explicit description, in terms of cardinal invariants,  of the ideal lattice of $\calE$ in   the bounded case. In  \cite{ASZ} the endomorphism ring of a bounded group is represented by a finite digraph whose vertices are the fundamental subgroups and arrows the irreducible endomorphisms. The  ideals are represented by  particular sets of paths in the digraph.

 In contrast, the approach in this paper is to use    the lattice $\calH$ of fully invariant subgroups of $G$ to   describe    certain sublattices of ideals of $\calE$. In the unbounded case, this approach fails to identify such important ideals as the various radicals of $\calE(G)$.

The functors $(\kappa,\,n)\to  p^\kappa G[p^n]$   for ordinals $\kappa$ and natural numbers $n$ defined in \S 3.1 and \S 3.2   for abelian $p$--groups $G$,  can also be applied to the ring $\calE$, even though the additive group of $\calE$ is not a $p$--group when $G$ is unbounded.  Namely, we define by induction $p^{\kappa+1}\calE=p(p^\kappa\calE)$ and $p^\kappa\calE=\bigcap_{\mu<\kappa}p^\mu\calE$ if $\kappa$ is a limit ordinal; and $\calE[p^n]=\{f\in\calE\colon p^nf=0\}$. Clearly, $p^\kappa\calE$ and $\calE[p^n]$ are closed under subtraction. Since endomorphisms do not decrease heights or increase exponents, it is also clear that $p^\kappa\calE$ and $\calE[p^n]$ and their intersection $p^\kappa\calE[p^n]$ are ideals of $\calE$.

In the definitions of  fundamental subgroups and indicator subgroups in $\calH$,  the ranks of the homocyclic components of  $G$ play no r\^ole. but
this is not true of the  ideals of the endomorphism ring $\calE$. For example, let $G$ have infinite  rank;   an endomorphism has rank $r$ if $\rank(Gf)=r$. Then the ring of endomorphisms of  finite rank is an ideal of $\calE$.

This is exemplified for bounded groups in  \cite{AvS}, where it is shown  that ideals of $\calE$ are parametrised by three variables: height, exponent and rank. More precisely, let  $G$ have exponent $e$,   homocyclic components $B_i,\  i\in[k]$ and infinite rank $\rho$. Let $k\in[e],\ n\in [e]$ and let $\mu$ an infinite cardinal $\leq\rho$. Let \[I(\kappa,\, n,\, \mu)=\{f\in\calE\colon Gf\leq p^\kappa
G[p^n], \rank(Gf)\leq\mu\}\]
Theorem 4.5 of   \cite{AvS} states that if $G$ is bounded, then for every ideal $I$ of $\calE,\ I=I(\kappa,\, n,\, \mu)$ for suitable choice of the parameters $(\kappa,\, n,\, \mu)$. The aim of this paper is to prove a similar result for unbounded groups.

\subsection{The lattices of fully invariant subgroups and ideals}

 Let $G$ be an unbounded group with endomorphism ring $\calE$.  Denote by $\calH$ and  $\calId $ the posets under inclusion of fully invariant subgroups of $G$   and   ideals of $\calE$ respectively.  It is straightforward to check that $\calH$ and $\calId$ are complete lattices under inclusion. 
  In this Section, I construct  a correspondence   ${}^\dagger\colon \calH \longleftrightarrow \calId$.

For all $H\in\calH$ and $I\in\calId$, let
  $H^\dagger=\{f\in\calE\colon Gf\leq H\}$ and $I^\dagger=\im(I)=\sum_{f\in I }\im f$.

\begin{lemma}\label{straight}   Let  $H\in\calH$ and $I\in\calId$.   Then
   $H^\dagger\in\calId$ and $I^\dagger\in\calH$.
\end{lemma}

\begin{proof}   Since $H$ is a fully invariant subgroup of $G$,    $H^\dagger$ is closed under addition and under left and right multiplication by elements of $\calE$, so is an ideal of $\calE$.

  Since $I$ is an ideal of $\calE$,      $I^\dagger$ is closed under addition, additive inverses 
 and  left and  right multiplication by members of $\calE$, so is a fully invariant subgroup of $G$.
\end{proof}
\begin{remark}\label{re} There is an essential difference between the functions $\dagger\colon \calH\to\calId$ and $\dagger\colon \calId\to\calH$: if $H,\, K\in\calH,\ H^\dagger=K^\dagger$ implies $H=K$ by definition, whereas if $I,\  J\in\calId,\  I^\dagger=J^\dagger$ does not imply any  other relation between $I$ and $J$   than the fact that they have the same image. For example, in our Example \ref{bex}, take $I$ to be multiplication by $p^3$ and $J$ to be the ideal generated by $f\colon a\mapsto pa,\ b\mapsto p^3b$. Then $I\not=J$ but $I^\dagger=G[p]=J^\dagger$.
\end{remark}
 
\begin{proposition}\label{Gal} With the notation above, 

\begin{enumerate}\item   The mappings  $H\mapsto H^\dagger$   and  $I\mapsto I^\dagger$ preserve  order  and arbitrary meets and joins.
 \item For all $H\in\calH$ and all   $I\in\calId$,
\begin{enumerate}
  \item $ H^{\dagger\dagger}\leq H$ and   $ I^{\dagger\dagger} \leq I$;
\item $H^{\dagger\dagger\dagger}= H^\dagger$ and $I^{\dagger\dagger\dagger} = I^\dagger$. 
\end{enumerate}\end{enumerate}
 \end{proposition}

\begin{proof} (1)    Let $H\leq K\in \calH$.  Then $H^\dagger \leq K^\dagger $. Similarly, if $L\leq J\in\calId$, then $I^\dagger\leq J^\dagger$.

Let $\calS\subseteq\calH$,  and $\calT\subseteq\calId$. Recall that $\sup(\calS)=\sum_{K\in\calS}K\in\calH,\ \inf(\calS)=\bigcap_{K\in\calS}K\in\calH,\  \sup(\calT)=\sum_{J\in\calT}J\in\calId$, and $\inf(\calT)=\bigcap_{J\in\calT}J\in\calId$.

 Let $K\in\calS$. Since  each  $K^\dagger\leq (\sup(\calS))^\dagger,\ \sup_{K\in\calS}(K^\dagger)\leq (\sup(\calS))^\dagger $.  Conversely, let $J\in\sup(S)$. Then there are finitely many $H_i\colon i\in[n]$
in $\calH$ such that $J=\sum_i H_i$ and hence $J^\dagger=\sum_iH_i^\dagger$. Consequently, $J^\dagger\in
\sup_{K\in\calS}(K^\dagger)$.

 Similarly,   $(\inf(\calS))^\dagger= \inf_{K\in\calS}(K^\dagger),\ (\sup(\calT))^\dagger= \sup_{I\in\calT}(I^\dagger)$, and $(\inf(\calT))^\dagger= \inf_{I\in\calT}(I^\dagger)$.

  (2) (a). Since $H^{\dagger\dagger}$ is the image of the ideal of all endomorphisms whose image is contained in $H,\  H^{\dagger\dagger}\leq H$.  Since $I^{\dagger\dagger}$ is the ideal of all endomorphisms which  map $G$ to $I^\dagger,\   I^{\dagger\dagger}\leq I$.

\quad (b) By (1)  and 2(a), $H^\dagger=H^{\dagger\dagger\dagger}$ and $I^\dagger=I^{\dagger\dagger\dagger}$.
\end{proof}

\begin{definition} Let $H\in\calH$ and $I\in\calId$. As in the references \cite{Birkhoff} and \cite{AbS},   we say $H $ ($I$) is \textit{$\dagger$--closed} if $H=H^{\dagger\dagger}$ ($I=I^{\dagger\dagger}$), 
\end{definition}
\begin{proposition}\label{wesay} \begin{enumerate}\item Let $H\in\calH$. The following are equivalent:
\begin{enumerate}\item there exists $I\in\calId$ such that $H=I^\dagger$;

\item  
 $H$ is  $\dagger$--closed;

\item there exists   a unique   $\dagger$--closed $J\in\calId$ such that $H=J^\dagger$.
\end{enumerate}

\item Let $I\in\calId$. The following are equivalent:
\begin{enumerate}\item there exists $H\in\calH$ such that  $I=H^\dagger$;

\item  
 $I$ is $\dagger$--closed;
\item there exists   a unique   $\dagger$--closed $   K\in\calH$ such that $I= K^\dagger$.
\end{enumerate} 
\end{enumerate}
\end{proposition}

\begin{proof} (1)  $(a) \Rightarrow (b)$ Let $I=H^\dagger $. Then $H^{\dagger \dagger } =I^\dagger =H$. 

$(b) \Rightarrow (c)$ Let $J=H^\dagger $. Then $J^\dagger =H^{\dagger \dagger }=H$, so $J^{\dagger \dagger }=H^\dagger =J$.
Suppose also $H=L^\dagger $ with $L$ closed. Then $L= L^{\dagger \dagger }=H^\dagger = J$.

$(c) \Rightarrow (a)$ is clear

(2)  The proofs of each part are strictly analogous to those of (1).
 \end{proof}

 \begin{notation} Denote by $\ov{\calH}$ and $\ov{\calId}$ the posets of   $\dagger$--closed subgroups of $G$ and ideals of $\calE$ respectively.
\end{notation}   

\begin{proposition} \label{proof}
\begin{enumerate}\item    $\ov{\calH}$ and 
$\ov{\calId}$ are complete lattices under inclusion;
\item $H\in  \ov{\calH}$ if and only if $ H^\dagger\in\ov{\calId}$; 
\item $I\in \ov{\calId}$ if and only if $ I^\dagger\in\ov{\calH}$;
 
 \item The mapping $H\mapsto H^\dagger\colon \ov{\calH}\to\ov{\calId}$ is a complete  lattice  isomorphism with inverse $I\mapsto I^\dagger\colon \ov{\calId}\to\ov{\calH}$.

\end{enumerate}
\end{proposition}

\begin{proof}(1) Since $\calH$ and $\calId$ are complete lattices, it suffices to show that their sublattices $  \calI^\dagger$ and $\ov{\calH}$ are closed under arbitrary meets and joins. But this follows from Proposition \ref{Gal} (1).

(2) and (3) These follow immediately from Proposition \ref{Gal} (2).

 (4) Let $\calS\subseteq\ov{\calH}$, let $U=\bigcap\calS$ and $V-\sum\calS$. Then $U^\dagger=\bigcap_{S\in\calS}S^\dagger$ and $V^\dagger=\sum_{S\in\calS } S^\dagger$.

 Similarly, if $\calT\subseteq \ov{\calId}$ let  $W=\bigcap\calT$ and $X-\sum\calT$.Then $W^\dagger=\bigcap_{T\in\calT}T^\dagger$ and $X^\dagger=\sum_{T\in\calT } T^\dagger$.
 
 This shows that  the mappings are complete lattice isomorphisms.\end{proof}

 The next task is to identify  the   $\dagger$--closed subgroups and ideals.   The case of $\ov{\calH}$ is straightforward:  we shall see that $H\in\calH$ is $\dagger$--closed if and only if   $H$ is an indicator subgroup.  However, $\ov{\calId}$ is more complicated; it will be shown that an ideal $I$ is $\dagger$--closed if and only if $I$ is maximal in the set of ideals having a given indicator subgroup as image.
 
 \subsection{$\dagger$--closed subgroups of $G$}

 I first show that fundamental subgroups are $\dagger$--closed. Recall  from Proposition \ref{wesay}, that for any fully invariant group or ideal $X$, if there exists an ideal $Y$ such that $X=Y^\dagger$ then $X$ is $\dagger$--closed.
 \begin{lemma}\label{fun} For  all   groups $G$ and all $n\in\bbN$,
 \begin{enumerate}\item   $p^n\calE^\dagger=p^nG$ and $p^nG^\dagger=p^n\calE$. 
 
 \item  $\calE[p^n]^\dagger=G[p^n]$ amd $G[p^n]^\dagger=\calE[p^n]$.
\end {enumerate}
 \end{lemma}
 
 \begin{proof} (1)  Let $a\in p^n\calE^\dagger$, so for all $f\in  p^n\calE$, there exists $b\in G$ such that $bf=a$.  In particular, take $f$ to be multiplication by $p^n$, so $a=p^nb\in p^nG$.  
  Conversely, let $a=p^nb\in p^nG$. Then for all $f\in\calE,\ af=b(p^nf)$ so $a\in p^n\calE^\dagger$.
  Consequently, $p^n\calE^\dagger=p^nG$.  Hence by Proposition \ref{wesay} (1) (c), $p^nG^\dagger=p^n\calE$.

 (2) Let $a\in\calE[p^n]^\dagger$. Then there exists $f\in\calE[p^n]$ and $b\in G$ such that $bf=a$. Since $p^nf=0,\ a\in G[p^n]$.  
  Conversely, let $a\in G[p^n]$. Then $G$ has a decomposition 
 $G=   B \oplus p^{n+1}G$ where $B$ is the maximal $p^n$--bounded pure subgroup of $G$ and $a\in B$. Let $f$ be the identity on $B$ and zero on $p^{n+1}G$. Then $f\in\calE[p^n]$ and $af=a$.
 Consequently, $\calE[p^n]^\dagger=G[p^n]$.  Hence by Proposition \ref{wesay}, (2) (c), $G[p^n]^\dagger= \calE[p^n]$. 
 \end{proof}
 
 Lemma \ref{fun}(1) can be extended to infinite ordinals as follows:
 \begin{proposition}\label{fun1} For all groups $G$ and and ordinals $\kappa,\ p^\kappa G^\dagger= p^\kappa\calE$ and $p^\kappa\calE^\dagger=p^\kappa G$.
 
 \end{proposition}
 \begin{proof} We have seen that both equalities hold for finite $\kappa$. Suppose both hold for all $\mu<\kappa$. If $\kappa$ is a limit, then $p^\kappa G^\dagger=\left(\bigcap_{\mu<\kappa}p^\mu G\right)^\dagger=  \bigcap_{\mu<\kappa}(p^\mu G)^\dagger = \bigcap_{\mu<\kappa}(p^\mu \calE)  =\left(\bigcap_{\mu<\kappa}p^\mu \calE) \right)=p^\kappa \calE$. That $p^\kappa\calE^\dagger= p^\kappa G$ follows from Proposition \ref{wesay}.
 
 If $\kappa=\mu+1$ then $\left(p^\kappa G\right)^\dagger=p(p^\mu G)^\dagger=p(p^\mu\calE)=p^\kappa\calE$. By Proposition \ref{wesay} again, $p^\kappa\calE^\dagger=p^\kappa G$.
 \end{proof}
\begin{corollary}\label{fgclosed} Fundamental subgroups and  indicator subgroups   are $\dagger$--closed.
\end{corollary}
\begin{proof} By Lemmas \ref{fun} and \ref{fun1},  all $G[p^n]$ and   $p^\kappa G$ are  $\dagger$--closed. Since the lattice of closed subgroups is closed under arbitary intersections and sums,  fundamental subgroups $p^\kappa G[p^n]$ are  $\dagger$--closed, and   for all admissible $\boldsigma$, the indicator group $G(\boldsigma)$ is   $\dagger$--closed.
 \end{proof}
\begin{corollary}\label{trgc} If $G$ is simply presented, then all fully invariant subgroups are $\dagger$--closed.
\qed\end{corollary}
  
%%%%%%%%%%%%%%%%%%%
\subsection{Dagger--closed ideals  of $\calE$} 

If    $\calE$ contains distinct ideals $I$ and $J$ such that $I^\dagger=J^\dagger$, by Proposition \ref{wesay}    not both $I$ and $J$ can  be $\dagger$--closed. So we first identify for which groups  this occurs and then determine how to identify $\dagger$--closed ideals. Some results are immediate: 
\begin{enumerate} \item  Let $G$ be homocyclic of exponent $n$ and   finite rank $m$,   then  $I\in\calId$ if and only if there exists $k\in[n]$ such that $I=p^k\calE$ and $I^\dagger=p^kG$. Hence $\calId$ is a chain of length $n+1$ and all ideals are $\dagger$--closed.

\item Let $G$ be homocyclic of exponent $n$ and infinite rank $\rho$.  Then 
for all integers $j\leq n$ and for all infinite cardinals $\mu\leq\rho$ the set \[I(j,\,\mu)=\{f\in p^j\calE\colon \rank(Gf)\leq\mu\}\]
is an ideal of $\calE$ whose image is $p^jG$. Hence if $\mu<\rho,\ I(j,\mu)$  is not $\dagger$--closed, but $I(j,\rho)$ is $\dagger$--closed.

\item Suppose $G$ of rank $\geq 2$ is not homocyclic. Then $G$ has a decomposition $G=\la a\ra\oplus\la b\ra\oplus C$ where $\exp(a)=n,\ \exp(b)=m$ with $0<n<m$.  There exist $f,\ g\in\calE$ such that $Cf=0=Cg,\ af=p^{m-1}	b=bf$,  while $ag=0$ and	$bg=p^{m-1}b$.

Let $I_f$ be the ideal generated by $f$ and $I_g$ the ideal generated by $g$. Then $I_f\ne I_g$ since they have different kernels, but $(I_f)^\dagger=\la p^{m-1}b\ra= (I_g)^\dagger$.

\end{enumerate}

In general then, if $G$ is not finite homocyclic, ideals of $\calE$ are not distinguished by their $\dagger$--closure. 
\begin{notation} Let $G$ be a group and $\calH$ the lattice of fully invariant subgroups of $G$. For all $H\in\calH$, the \textit{$\dagger$--inverse of $H$}, 
$ \di(H)=\{I\in\calId\colon I^\dagger=H\}$.
\end{notation}

 \begin{lemma}\label{edagclass} For all $H\in\calH,\ \di(H)$ is closed under sums  and $\sum_{I\in \di(H)} I$  is the only $\dagger$--closed ideal in $\di(H)$.
\end{lemma} 
 
\begin{proof}   Let $S\subseteq \di(H)$. Since   $\sum S$ is an ideal and has image $H$,  so $\sum S\in   \di(H)$.

 In particular, if $S=\di(H)$,  then $H^\dagger=\sum S$ so   $\sum S$ is $\dagger$--closed and so is the only $\dagger$--closed  ideal in $\di(H)$.
\end{proof}

\section{The ideal lattice of $\calE$}

Let $G$ be an unbounded group of length $\lambda$ and rank $\rho$ with endomorphism ring $\calE$. Recall that for $f\in\calE,\ \rank(f)$ means $\rank(Gf)$. Let $U$ be the increasing sequence of  infinite cardinals $\mu\leq\rho$ and all ideals $I$ of $\calE$ and $\mu\in U$, let 
$\calE_{\leq\mu}=\{f\in\calE\colon \rank(f)\leq \mu\}$, and $I_{\leq\mu}=I\bigcap\calE_{\leq\mu}$.

\begin{lemma}\label{frank}\begin{enumerate}\item For all ideals $I$ of $\calE$ and all infinite $\mu\leq\rho,\ I_{\leq\mu}$ is an ideal of $\calE$;
\item If $\mu_1<\mu_2<\cdots <\mu_\kappa$ is a chain of cardinals, and $I$ is an ideal of rank $\geq \mu_\kappa$ then $I_{\leq\mu_1}<I_{\leq\mu_2}<\cdots <I_{\leq\mu_\kappa}$ is a chain of distinct ideals of $\calE$.
\end{enumerate}\end{lemma}
\begin{proof} (1 ) Since $\calE_{\leq\mu}$ is closed under addition and left and right multiplication by elements of $\calE$ it is an ideal. Consequently, so is $I_{\leq\mu}$.

(2) 
Let $f\in\calE_{\leq \mu_{i+1}}\setminus \calE_{\leq\mu_i}$.
Then $If\leq I_{\leq\mu_{i+1}}\setminus I_{\leq\mu_i}$.
\end{proof}

 \begin{proposition}\label{clas} Let $G$ be unbounded of length $\lambda$ and let $H$ be a fully invariant subgroup. Suppose $\rank(H)=\rho$ is infinite and let $U$ be the increasing sequence  of infinite cardinals $\leq\rho$.   For all $\mu\in U$,  let ${H^\dagger}_{\leq\mu}=\{f\in H^\dagger\colon \rank Gf\leq\mu\}$. Then
\begin{enumerate}\item  For all $\mu\in U$, ${H^\dagger}_{\leq\mu}$ is an ideal of $\calE$ contained in $H^\dagger$;
\item  $\kappa<\mu$ in $U$ implies ${H^\dagger}_{\leq\rho}\leq {H^\dagger}_{\leq\mu}$;
\item ${H^\dagger}_{\leq\rho}$ is the only $\dagger$--closed ideal in $\di(H)$.
\end{enumerate}\end{proposition}

\begin{proof} (1) and (2)  follow immediately from the definition.

(3) By definition, ${H^\dagger}_{\leq\rho}$ is the ideal of all endomorphisms mapping $G$ into $H$, that is, $H^\dagger$. By Proposition \ref{wesay}, ${H^\dagger}={H^\dagger}_{\leq\rho}$ is the unique $\dagger$--closed ideal in $\di(H)$.
\end{proof}
\begin{remarks} Unlike in the bounded case, it is an open problem whether there are ideals in $\di(H)$ other than the $I_{\leq\mu}$ for some infinite cardinal $\mu$, even in the case that $H$ is a fundamental group.   \end{remarks}

 For the statement of our final result, recall the notation:
\begin{itemize}\item $G$ is an unbounded  group of   length $\lambda$ and rank $\rho$;
\item   $\calH$ is the lattice of fully invariant subgroups of $G$ and $\ov{\calH}$ is the lattice of $\dagger$--closed elements of $\calH$;
\item   $\calId $ is the lattice of ideals of $\calE$ and $\ov{\calId}$ is the lattice of $\dagger$--closed ideals;
\item $U$ is the increasing sequence of infinite cardinals $\leq\rho$; for all $I\in\calId$ and all $\mu\leq\rho,\ I_{\leq\mu}=\{f\in I\colon \rank(Gf)\leq\mu\}$, so $I=I_\rho$;
\item  $\calI=\calI(G)$ is the lattice of $G$--admissible indicators and 
  for all $\boldsigma\in\calI,\ G(\boldsigma)=\{a\in G\colon \ind(a)\geq\boldsigma\}$ and $\rank(G(\sigma))=  \rho_\sigma$;
\end{itemize}
\begin{theorem}\label{Main} Let $G$ be an unbounded group of length $\lambda$ and rank $\rho$. 
\begin{enumerate}
\item $\ov {\calId}=\{H^\dagger\colon H\in\calH\}$  and $I\mapsto I^\dagger\colon \ov{\calId}\to \ov{\calH}$ is a lattice  isomorphism;
\item   For all $\boldsigma\in \calI(G)$, let $I(\boldsigma ) =  G(\boldsigma)^\dagger$. 
Then $\{I(\boldsigma)\colon\boldsigma\in \calI(G)\}$   is a  sublattice   of $\ov{\calId}$ and the isomorphism of $(1)$ maps  $\{I(\boldsigma)\colon\boldsigma\in \calI(G)\}$ to the lattice    of indicator subgroups of $G$.
\item For each $H\in\calH$,   let $(\mu_1<\mu_2<\cdots<\mu_\kappa=\rho_\sigma)\colon \kappa\in K$ be a sequence of  infinite cardinals $\leq\rho_H$. Then $\left(H^\dagger_{\leq\mu_k}\colon k\in K\right)$ is a chain of distinct ideals of $\calE$ contained in $H^\dagger$.

\item For all $H,\,H'\in\calH,  H^\dagger_{\leq \mu}\leq {H'}^\dagger_{\leq \mu'}$ if and only if $H\geq H'$ and $\mu\leq \mu'$;

\end{enumerate}
\end{theorem} 
 
\begin{proof} (1) By Proposition \ref{wesay} (2). $H^\dagger$ is a $\dagger$--closed ideal of $\calE$, and by Proposition \ref{proof} (4) the map is a lattice isomorphism.

(2) This follows from Proposition \ref{clas}.

(3) This follows from Lemma \ref{frank}.

(4) This follows immediately from the definitions.
\end{proof}

\begin{remark} If $G$ is simply presented, then we have seen that every fully invariant subgroup is $ G(\boldsigma)$ for some $\boldsigma\in\calI(G)$ and hence $\calH=\ov{\calH}$. Hence Parts (1) and (2) of Theorem \ref{Main} coalesce.  

 If $G$ is not simply presented, then $\calE$ and the   ideal lattice  $\ov {\calId}$ and  its sublattice depend on the properties of groups lying between the derived basic subgroups $B^{(\xi)}$ and their $p$--adic closures $\ov{B^{(\xi)}}$ of \S 5, about which little is currently known.

\end{remark}


\begin{thebibliography}{99}

\bibitem
[\textbf{AbS}, 2004]{AbS}
{R, Abraham and   P. Schultz }, \textit{Aditive Galois Theory of Modules}, in Rings, Modules, Algebras and Abelian Goups, Vol. 236,  Lecture Notes in Pure and Applied Mathematics, Marcel Dekker Inc., (2004)


\bibitem [\textbf{AvS}, 2000]{AvS} M. A. Avino,\ Phill Schultz,\textit{The endomorphism ring of a bounded abelian $p$--group}, in: Abelian groups, rings  and modules, Contemp. Math., 273, Amer. math. Soc., (2001), 75--84
 

\bibitem
[\textbf{ASZ}, 2020]{ASZ}
{M. A. Avi\~no,\ Phill Schultz and Marcos Zyman}, \textit{The upper central series of the maximal normal $p$--subgroup of a group of automorphisms}, Journal of Group Theory, 24(6), 1213-1244, (2020)
\bibitem
[\textbf{Birkhoff}, 1967]{Birkhoff} G.. Birkhoff, \textit{Lattice Theory},  A.M.S. Colloquium Publ., Vol.26, Third Ed., (1967).
 .

 
\bibitem
[\textbf{Corner}, 1969]{Corn} A. L. S. Corner, \textit{On endomorphism rings of primary abelian groups}, Quart. J. Math. Vol. 20, No. 79 (1969), 277-=296

 

\bibitem[\textbf{Fuchs}, 2015]{Fuchs}
{ L.~Fuchs}, \textit{Abelian Groups}, Springer Monographs in Mathematics, Springer 2015. 


\bibitem[\textbf{Hausen}, 1982]{H82}
{J. Hausen}, \textit{ Infinite general linear groups over rings},
Archjv der Math., 39, (1982) 510--524.


\bibitem[\textbf {Jacoby and Loth, 2019}] {J-L}  C. Jacoby and P. Loth, \textit{Abelian Groups: Structures and Classificatioins} De Gruyter Studies in Mathematics, 73 (2019).

\bibitem[\textbf {Krylov, Mikhalev and Tuganbaev, 2003}] {KMT}  P. . Krylov, A. V. Mikhalev and A. A. Tuganbaev, \textit{Endomorphism Rings of Abelian Groups} KluwernAcademic Publishers (2003).


\bibitem[\textbf{Pierce}, 1963]{Pierce}
{ R. S. Pierce}, \textit{Homomorphisms of primary abelian groups}, in \textbf{Topics in Abelian Groups}Scott, Foresman and Co., 1963, 215--310.

\bibitem[\textbf{Stringall, 1967}]{String}\textit{Endomorphism rings of primary abelian groups} 
Pac, J.  Math.   20,  . 3, (1967), 535--557.


\end{thebibliography}
\end{document}